\documentclass[reqno]{amsart}
\usepackage[english]{babel}
\usepackage{amscd,amssymb,amsmath,amsfonts,latexsym,amsthm}
\usepackage{ulem}
\usepackage{inputenc}
\usepackage{graphicx,color}
\usepackage{tabularx}
\newtheorem{thm}{Theorem}[section]

\newtheorem{lem}[thm]{Lemma}

\newtheorem{rem}[thm]{Remark}

\numberwithin{equation}{section}
\newcommand{\C}{\mathbb{C}}

\newcommand{\R}{\mathbb{R}}
\newcommand{\N}{\mathbb{N}}
\newcommand{\mg}{\mathfrak{g}}
\newcommand{\mr}{\mathfrak{r}}
\newcommand{\ms}{\mathfrak{s}}

\newcommand{\Sl}{\mathfrak{sl}}
\newcommand{\so}{\mathfrak{so}}
\newcommand{\mh}{\mathfrak{h}}
\newcommand{\mn}{\mathfrak{n}}

\title[Cohomological rigidity of conformal Galilei algebras and their central extensions]{Cohomological rigidity of conformal Galilei algebras and their central extensions}

\author[A. Khudoyberdiyev]{Abror Khudoyberdiyev$^{a,b}$}
	
\author[D. Jumaniyozov]{Doston Jumaniyozov$^{b,a}$}

\address{$^a$V.I.Romanovskiy Institute of Mathematics Uzbekistan Academy of Sciences.}

\address{$^b$National University of Uzbekistan named after Mirzo Ulugbek.}

\email{{\tt khabror@mail.ru,  dostondzhuma@gmail.com}}

\begin{document}
\date{}

\begin{abstract}
In this paper, we study the second adjoint cohomology of the compexification of the real conformal Galilei algebras \(\mathfrak{cga}_\ell(d,\mathbb{R})\) and their central extensions. These algebras are non-semisimple Lie algebras that appear as non-relativistic analogues of conformal Lie algebras. Using cohomological methods, including the Hochschild-Serre factorization theorem, we compute the space \(H^2(\mathfrak{g}, \mathfrak{g})\) and examine conditions under which these algebras are cohomologically rigid. Our main result shows that the conformal Galilei algebras are rigid for all spatial dimensions \(d \neq 2\) when the spin \(\ell\) is a half-integer. This provides a complete characterization of the formal rigidity of these algebras in the specified parameter range.


\end{abstract}

\maketitle

\thispagestyle{empty}

\section{Introduction}

The study of symmetries plays a fundamental role in both mathematics and theoretical physics. In the context of partial differential equations on manifolds-particularly those of geometric origin-the associated symmetry algebra provides deep structural information. It enables the derivation of new solutions from known ones and reflects essential geometric properties of the underlying manifold. Among such systems, the free Schr\"odinger equation stands out for its foundational role in quantum mechanics. Its symmetry algebra, known as the Schr\"odinger algebra, belongs to a broader class of non-relativistic symmetry algebras called conformal Galilei algebras, which were first systematically analyzed in \cite{Havas-Plebanski1978}.

One of the most powerful tools for understanding the structure and classification of Lie algebras is cohomology. Introduced through the foundational works of Cartan, Chevalley, Eilenberg, Koszul, Hochschild and Serre \cite{ChEg, Ksz, HchSe}, Lie algebra cohomology has deep connections to deformation theory, extension theory, and geometric representation theory. In particular, the second cohomology group with coefficients in the adjoint module plays a central role in determining whether a Lie algebra is formally rigid-that is, whether it admits nontrivial infinitesimal deformations. For instance, the classical result in algebraic geometry asserts that every algebraic variety-such as the variety formed by algebras defined through identities-can be decomposed into a finite union of irreducible components. In this context, rigid algebras are characterized by having open orbits under the natural action of the general linear group in the Zariski topology. It is also known that Lie algebras with vanishing second cohomology in the adjoint representation are rigid, and the closures of their orbits correspond to irreducible components of the variety (see \cite{NR}, Theorem~7.1). This insight has inspired a significant body of work aimed at identifying algebras with open orbits and understanding their structural properties \cite{Ru, Gest, NR}.

While semisimple Lie algebras demonstrate cohomological rigidity-thanks to vanishing theorems such as Whitehead’s Lemmas-the cohomological structure of non-semisimple Lie algebras is significantly more intricate. Among these, a particularly important class is that of perfect Lie algebras, characterized by the condition that the algebra is equal to its own derived subalgebra, i.e., $\mathfrak{g} = [\mathfrak{g}, \mathfrak{g}]$. Semisimple Lie algebras are classical examples, but many non-semisimple Lie algebras arising in mathematical physics are also perfect. For such algebras, the second cohomology group often does not vanish, and understanding its structure becomes essential in exploring their deformation theory and classifying modules of Lie algebra structures.
Numerous studies have focused on the cohomology of perfect Lie algebras (see \cite{BU75},\cite{Ru},\cite{DJuSSh}). Namely, the cohomological rigidity of Schr\"odinger algebras is examined in \cite{Ru}, while the infinitesimal deformations of the $n$-th Schr\"odinger algebra are investigated in \cite{DJuSSh}.

In this paper, we investigate the second adjoint cohomology of a class of non-semisimple perfect Lie algebras, with particular emphasis on conformal Galilei algebras and their central extensions. Our aim is to analyze the conditions under which these algebras exhibit cohomological rigidity and to describe the structure of their infinitesimal deformations. By doing so, we contribute to the broader effort of understanding the deformation theory of non-semisimple perfect Lie algebras and their role in mathematical physics.

The structure of the paper is as follows. In Section 2, we present essential background information on Lie algebra cohomology, including the Hochschild-Serre factorization theorem. Section 3 offers a brief introduction to the conformal Galilei algebras. Section 4 is devoted to computing the second cohomology spaces of these algebras with respect to the adjoint representation. Lastly, in Section 5, we investigate the second cohomology groups of the mass central extensions of the conformal Galilei algebras.

\section{Preliminaries}

Let $\mathfrak{g}$ be a Lie algebra. A $\mathfrak{g}$-module is a vector space $M,$ together with a homomorphism of $\mathfrak{g}$ into the Lie algebra of all linear transformations of $M.$ We denote by $v.m$ the transform of the element $m\in M$ by the linear transformation which corresponds to the element $v\in\mathfrak{g}.$ The subspace of $M$ which consists of all $m\in M$ with $v.m=0,$ for all $v\in\mathfrak{g}$ is denoted by $M^{\mathfrak{g}}.$
Set
$$C^0(\mathfrak{g},M)=M, \quad C^n(\mathfrak{g},M)={\rm Hom}(\wedge^{n}\mathfrak{g},M), \ n>0.$$
That is $C^n(\mathfrak{g},M)$ is the space of alternating $n$-linear maps. The elements of $C^n(\mathfrak{g},M)$ are called $n$-{\it cochains}. Now, define the differential $d^n:C^n(\mathfrak{g},M)\to C^{n+1}(\mathfrak{g},M).$ Let $\varphi\in C^n(\mathfrak{g},M),$ be an $n$-cochain. The differential $d^n\varphi\in C^{n+1}(\mathfrak{g},M)$ of $\varphi,$ is defined as follows:
\begin{equation} \label{eq1} \begin{array}{ll}
(d^n\varphi)(e_0,\dots,e_n)&=\sum\limits_{i=0}^{n+1}(-1)^{i}e_i.\varphi(e_0,\dots,\hat{e}_i,\dots,e_n)\\[3mm]
&+\sum\limits_{1\leq i<j\leq n}(-1)^{i+j}\varphi([e_i,e_j],e_0,\dots,\hat{e}_i,\dots,\hat{e}_j,\dots,e_n),
\end{array}
\end{equation}
 where $e_0,\dots,e_n\in\mathfrak{g}$ and the sign $\hat{}$ indicates that the argument below it must be omitted. Set $Z^{n}(\mathfrak{g},M):={\rm Ker}(d^{n+1}),$ $B^{n}(\mathfrak{g},M):={\rm Im}(d^{n}).$ The property $d^{n+1}\circ d^{n}=0$ leads that the derivative operator $d=\sum\limits_{i\geq0}d^i$ satisfies $d\circ d=0.$ Therefore, the $n$-th cohomology group is well defined by
$$H^{n}(\mathfrak{g},M)=Z^{n}(\mathfrak{g},M)/B^{n}(\mathfrak{g},M).$$
The elements of $Z^{n}(\mathfrak{g},M)$ and $B^{n}(\mathfrak{g},M)$ are called $n$-{\it cocycles} and $n$-{\it coboundaries}, respectively.

If $\mathfrak{r}$ is an ideal of $\mathfrak{g},$ then we define on each $C^n(\mathfrak{r},M)$ the structure of $\mathfrak{g}$-module. Set $C^{0}(\mathfrak{r},M)=M,$ and thus is already defined as a $\mathfrak{g}$-module. For $n>0,$ $\omega\in C^n(\mathfrak{r},M),$ $v\in\mathfrak{g},$ $e_1,\dots,e_n\in\mathfrak{r},$ we define the transform $v.\omega$ by the following formula:
\begin{align}{\label{action}}
(v.\omega)(e_1,\dots,e_n)=v.\omega(e_1,\dots,e_n)-\sum_{i=1}^{n}\omega(e_1,\dots,e_{i-1},[v,e_{i}],e_{i+1},\dots,e_n).
\end{align}

In general, the computation of cohomology spaces is complicated. However, if a Lie algebra is the semidirect product of a semisimple subalgebra and its radical, the Hochschild-Serre factorization theorem (Theorem 13 in \cite{HchSe}) simplifies this computation. Namely, let $\mathfrak{g}=\mathfrak{s}\ltimes\mathfrak{r}$ be a Lie algebra, where $\mathfrak{s}$ is the semisimple part of $\mathfrak{g}$, $\mathfrak{r}$ is its radical and $M$ is a finite-dimensional $\mathfrak{g}$-module. Then one has  the following isomorphisms of vector spaces:
$$H^p(\mathfrak{g},M)\cong\sum_{m+n=p}H^m(\mathfrak{s},\mathbb{C})\otimes H^n(\mathfrak{r},M)^{\mathfrak{s}}.$$
Here, $H^q(\mathfrak{r},M)^{\mathfrak{s}}$ is the space of $\mathfrak{s}$-invariant cocycles of $\mathfrak{r}$ with values in $M$ which is defined as
$$H^q(\mathfrak{r},M)^{\mathfrak{s}}=\{[\varphi]\in H^q(\mathfrak{r},M) \ | \ \forall v\in\mathfrak{s}:  v.\varphi\in B^q(\mathfrak{r},M)\},$$
where $[\varphi]$ denotes the cohomology class of the $q$-cocycle $\varphi$ and the action of $\mathfrak{g}$ on $Z^q(\mathfrak{r},M)$ is defined as in \eqref{action}. Namely, since the $\mathfrak{g}$-action on $C^q(\mathfrak{r},M)$, as defined in \eqref{action}, is compatible with the differential, $Z^q(\mathfrak{r},M)$ and $B^q(\mathfrak{r},M)$ are invariant under the action of $\mathfrak{s}$, and therefore this action can be lifted to $H^q(\mathfrak{r},M)$.
Note that the following isomorphism holds for $H^q(\mathfrak{r},M)^{\mathfrak{s}}$:
\begin{align}{\label{iso23}}
\nonumber
H^q(\mathfrak{r},M)^{\mathfrak{s}}&=[Z^q(\mathfrak{r},M)^{\mathfrak{s}}+B^q(\mathfrak{r},M)]/B^d(\mathfrak{r},M)\\
                                            &\simeq Z^q(\mathfrak{r},M)^{\mathfrak{s}}/[B^q(\mathfrak{r},M)\cap Z^q(\mathfrak{r},M)^{\mathfrak{s}}]\\
                                            \nonumber
                                            &= Z^q(\mathfrak{r},M)^{\mathfrak{s}}/B^q(\mathfrak{r},M)^{\mathfrak{s}}.
\end{align}
The first equality is based on the following fact: Since $Z^d(\mathfrak{r},M)$ is finite-dimensional, $Z^d(\mathfrak{r},M)$ is completely reducible as an $\mathfrak{s}$-module. Therefore, for any cocycle $\varphi\in Z^d(\mathfrak{r},M)$ there exists an $\mathfrak{s}$-invariant cocycle $\varphi^{\prime}$ such that $\varphi-\varphi^{\prime}\in B^d(\mathfrak{r},M)$ (see Proposition 2.1 in \cite{WuZh}).

We recall the following from \cite{BU75}. Let $\mg=\ms\ltimes V$, where $\ms$ is semisimple and $V$ is an $\ms$-module.
Consider $V$ as an abelian Lie algebra. Then we have a short exact sequence of $\mg$-modules
\[
0\to V\to \mg\to \mg/V\to 0.
\]
This is also a short exact sequence of $V$-modules by restriction to $V\subseteq \mg$. Here $V$ and $\mg/V$ are trivial
modules. This yields a long exact sequence in cohomology, because the functor of $\ms$-invariants is exact on finite-dimensional
modules.

\begin{lem}\label{2.1}
\cite{BU75} Let $\mg=\ms\ltimes V$, where $\ms$ is semisimple and $V$ is an $\ms$-module. Then we have
\begin{align*}
0 & \to H^0(V,V)^\ms\to H^0(V,\mg)^\ms\to H^0(V,\mg/V)^\ms \\
  & \to H^1(V,V)^\ms\to H^1(V,\mg)^\ms\to H^1(V,\mg/V)^\ms \\
  & \to H^2(V,V)^\ms\to H^2(V,\mg)^\ms\to H^2(V,\mg/V)^\ms \\
  & \to H^3(V,V)^\ms\to H^3(V,\mg)^\ms\to H^3(V,\mg/V)^\ms \to \cdots
\end{align*}
The sequence continues until $k> \dim V$, in which case we have
\[
H^k(V,V)^{\ms}=H^k(V,\mg)^\ms=H^k(V,\mg/V)^\ms=0.
\]
\end{lem}

For later use, we recall the well known Schur's lemma.

\begin{lem}
\cite{HU78}
Let $V$ be a finite dimensional irreducible $\mg$-module over an algebraically closed field $k$, and $\varphi : V \to V$ is an
$\mg$-homomorphism. Then $\varphi = \lambda \cdot Id $ for some $\lambda\in k.$

\end{lem}

Below we present a generalized version of the Schur's lemma.

\begin{lem}
\cite{KI2008}
Let $\mg$ be a complex Lie algebra, let $\mathcal{V} = \mathcal{V}_1 \oplus \dots \oplus \mathcal{V}_m$ and $\mathcal{W} = \mathcal{W}_1 \oplus\dots\oplus \mathcal{W}_n$ be completely reducible $\mg$-modules, where $\mathcal{V}_1, \dots , \mathcal{V}_n, \mathcal{W}_1, \dots , \mathcal{W}_n$ are irreducible
submodules. Then any $\mg$-module homomorphism $\varphi : \mathcal{V} \to \mathcal{W}$ can be represented as
$$\varphi=\sum_{i=1}^{m}\sum_{j=1}^{n}\lambda_{ij}\varphi_{ij},$$
where the operator $\varphi_{ij} : \mathcal{V}_i \to \mathcal{W}_j$ are fixed $\mg$-module homomorphisms and $\lambda_{ij}$ are complex
numbers. Furthermore $\varphi_{ij}\neq 0$ if and only if $\varphi_{ij}$ is an isomorphism if and only if $\mathcal{V}_i$ and $\mathcal{W}_j$ are isomorphic as $\mg$-modules.
\end{lem}

\section{Conformal Galilei algebras}

The conformal Galilei algebras \(\mathfrak{cga}_\ell(d, \mathbb{R})\) are non-relativistic generalizations of the conformal Lie algebras \(\mathfrak{so}(d+1, 2, \mathbb{R})\),
classified by the pair \((d, \ell)\), where \(d \in \mathbb{N}\) denotes the spatial dimension and \(\ell \in \tfrac{1}{2} \mathbb{N}_0\) represents the spin.
These algebras are not semisimple, and they admit a one-dimensional central extension either when \(d \in \mathbb{N}\) and \(\ell\) is a half-integer, or when \(d = 2\) and \(\ell\) is a non-negative integer.

Let us consider the \((d+1)\)-dimensional space \(\mathbb{R}^{1,d}\), equipped with the standard metric of signature \((1,d)\). We denote the standard linear coordinates on \(\mathbb{R}^{1,d}\) by \((t, x_1, x_2, \dots, x_d)\), where \(t\) represents the time coordinate and \(x_1, x_2, \dots, x_d\) are the spatial coordinates. We now focus on the infinitesimal transformations of \(\mathbb{R}^{1,d}\), described by vector fields
\begin{align}\label{eq:cga relations}
\begin{gathered}
  H=\partial_t, \qquad D=-2t\partial_t-2\ell\sum_{i=1}^d x_i \partial_{x_i}, \qquad C=t^2\partial_t + 2\ell \sum_{i=1}^d tx_i\partial_{x_i} \\
  M_{i,j}=-x_i\partial_{x_j}+x_j \partial_{x_i}, \qquad P_{n,i}=(-t)^n\partial_{x_i},
\end{gathered}
\end{align}
where $i,j=1,2\dots,d$ and $n=0,1,\dots,2\ell$ with $\ell \in \frac{1}{2}\N_0$. Their linear span has real
dimension $\smash{\frac{d(d-1)}{2}}+(2\ell+1)d + 3$ and is closed under the Lie bracket of vector fields. Consequently, we get finite-dimensional real Lie algebra with non-trivial Lie brackets
\begin{align*}
[M_{i,j},M_{k,r}]&=-\delta_{i,k}M_{j,r}-\delta_{j,r}M_{i,k}+ \delta_{i,r}M_{j,k} + \delta_{j,k}M_{i,r}, \\
[D,H]&=2H,  \qquad \qquad [C,H]=D,  \qquad \qquad [D,C]=-2C,\\
  [M_{i,j},P_{n,k}]&=-\delta_{i,k}P_{n,j}+\delta_{j,k}P_{n,i},\\
  [H,P_{n,i}]&=-nP_{n-1,i}, \qquad [D,P_{n,i}]=2(\ell-n)P_{n,i}, \qquad [C,P_{n,i}]=(2\ell-n)P_{n+1,i},
\end{align*}
where $i,j=1,2\dots,d$ and $n=0,1,\dots,2\ell$. The Lie algebra $\mathfrak{cga}_\ell(d,\R)$ has a Lie subalgebra given by direct sum of $\mathfrak{sl}(2,\R) \simeq \mathfrak{so}(2,1,\R)$ generated by $C,D,H$ and $\mathfrak{so}(d,\R)$ generated by $M_{i,j}$ for $i,j=1,2,\dots,d$. The subspace generated by $P_{n,i}$ for $i=1,2,\dots,d$ and $n=0,1,\dots,2\ell$ forms an abelian ideal and we have $\mathfrak{cga}_\ell(d,\R) \simeq (\mathfrak{sl}(2,\R) \oplus \mathfrak{so}(d,\R)) \ltimes V_{2\ell \omega,\omega_1}$, where $V_{2\ell \omega,\omega_1}$ is the irreducible finite-dimensional $\mathfrak{sl}(2,\R) \oplus \mathfrak{so}(d,\R)$-module with highest weight $(2\ell \omega,\omega_1).$

The conformal Galilei algebras $\mathfrak{cga}_\ell(d,\R)$ admit two distinct types of central extensions according to the values of $d\in \N$ and $\ell\in \frac{1}{2}\N_0$. In particular, we have
\begin{enumerate}
  \item[1)] the mass central extension $\widehat{\mathfrak{cga}}_\ell(d,\R)$ for $d \in \N$ and $\ell \in \frac{1}{2}+\N_0$,
  \begin{align}\label{3.2}
    [P_{m,i},P_{n,j}]=\delta_{i,j}\delta_{m+n,2\ell}b_m M, \qquad b_m=(-1)^{m+\ell+\frac{1}{2}}(2\ell-m)!m!,
  \end{align}
  \item[2)] the exotic central extension $\widetilde{\mathfrak{cga}}_\ell(d,\R)$ for $d=2$ and $\ell \in \N_0$,
  \begin{align}
    [P_{m,i},P_{n,j}]=\varepsilon_{i,j}\delta_{m+n,2\ell}q_m \Theta, \qquad q_m=(-1)^m(2\ell-m)!m!,
  \end{align}
  where $\varepsilon_{1,2}=-\varepsilon_{2,1}=1$ and $\varepsilon_{1,1}=\varepsilon_{2,2}=0$.
\end{enumerate}

In the rest of the article we shall consider the complexification of the real conformal Galilei algebras
$\mathfrak{cga}_\ell(d,\C).$
The expression $\langle x_1, \dots, x_n \rangle$ denotes the linear subspace of the Lie algebra $\mathfrak{g}$ generated by the elements $x_1, \dots, x_n \in \mathfrak{g}$; that is, the smallest vector subspace of $\mathfrak{g}$ containing all of these elements.

\section{The second adjoint cohomology space of $\mathfrak{cga}_\ell(d,\C)$}

In this section, we present the results regarding the second adjoint cohomology of the complex conformal Galilei algebra for $l\in\frac{1}{2}+\N_0.$ Namely, we show that the second adjoint cohomology vanishes for $l\in\frac{1}{2}+\N_0, \ d\neq2$ by applying Hochschild-Serre factorization theorem. For convenience, let us use the denotations $\mg=\mathfrak{cga}_\ell(d,\C), \ \ms=\mathfrak{sl}_{2}(\C) \oplus \mathfrak{so}_d(\C)$ and $V=V_{2\ell \omega,\omega_1}.$

\subsection{The case $d\geq3$ and $l\in\N_0$}

First, we state the main theorem and then we prove several auxiliary lemmas. At the end of  the section, we present the proof of the main theorem.

\begin{thm}{\label{mainthm}} Let $\mg=\mathfrak{cga}_\ell(d,\C)$ and $d\geq3.$ Then we have
\begin{align*}
H^2(\mg,\mg)&=0, \ \mbox{ if } \ l\in\frac{1}{2}+\N_0,\\
H^2(\mg,\mg)&\cong\C, \ \mbox{ if } \ l\in\N_0.
\end{align*}
\end{thm}

At first, we prove that the space of $\ms$-invariant coboundaries vanishes.

\begin{lem}{\label{mainlem1}}
Let $\mg=\mathfrak{cga}_\ell(d,\C)$ and $d\geq3.$
$ B^2(V,\mg)^{\ms}=0.$
\end{lem}
\begin{proof}
First, we show that $\omega(P_{i,j})\in V, \ i\in\{0,1,\dots,2l\}, \ j\in\{1,2,\dots,d\}.$ Let $\omega$ be an $\ms$-invariant 1-cochain. Then the invariance with respect to $D\in\Sl(2,\C)$ leads to
\begin{align*}
(D.\omega)(P_{i,j})=[D,\omega(P_{i,j})]-\omega([D,P_{i,j}])=[D,\omega(P_{i,j})]-2(l-i)\omega(P_{i,j})=0.
\end{align*}
So, $\omega(P_{i,j})$ is an eigenvector of ${\rm ad}_D$ with eigenvalue $2(l-i).$ Thus, we have the following relations:
\begin{align*}
\omega(P_{l-1,j})&\in\langle H,P_{l-1,j} \ | \ j=1,2,\dots,d \rangle, &  \omega(P_{l+1,j})&\in\langle C,P_{l+1,j} \ | \ j=1,2,\dots,d \rangle,\\
\omega(P_{l,j})&\in\langle M_{p,q},P_{l,j} \ | \ p,q,j=1,2,\dots,d \rangle, & \omega(P_{i,j})&\in V, \quad i\neq l-1,l,l+1.
\end{align*}
Moreover, the invariance relations $H.\omega=0$ imply that
\begin{align*}
(H.\omega)(P_{l,j})&=[H,\omega(P_{l,j})]-\omega([H,P_{l,j}])=[H,\omega(P_{l,j})]+l\omega(P_{l-1,j})=0.
\end{align*}
Hence, $\omega(P_{l-1,j})\in\langle P_{l-1,j} \ | \ j=1,2,\dots,d\rangle.$ Similarly, the invariance with respect to $C\in\Sl(2,\C)$ follows that $\omega(P_{l+1,j})\in\langle P_{l+1,j} \ | \ j=1,2,\dots,d\rangle$ and $\omega(P_{l,j})\in\langle P_{l,j} \ | \ j=1,2,\dots,d\rangle.$ Thus, we have $\omega:V\to V$ and $x.\omega=0$ for $x\in\ms,$ that is,
$$(x.\omega)(v)=[x,\omega(v)]-\omega([x,v])=0, \quad v\in V.$$
This means that $\omega$ is an $\ms$-module homomorphism. Then by Schur's lemma we have
$$\omega=\alpha{\rm Id}_{|V}, \quad \alpha\in\C.$$
Since $V$ is abelian, ${\rm Id}_{V}$ is a derivation of $V.$ Hence, $B^2(V,\mathfrak{cga}_\ell(d,\C))^{\ms}=0.$
\end{proof}

\begin{lem}{\label{mainlem2}}
Let $d\geq3.$ Then we have
\begin{align*}
Z^2(V,\mathfrak{so}_d(\C))^{\ms}&=0, \quad \mbox{ if } \ l\in\frac{1}{2}+\N_0,\\
Z^2(V,\mathfrak{so}_d(\C))^{\ms}&\cong\C, \quad \mbox{ if } \ l\in\N_0.
\end{align*}
\end{lem}

\begin{proof}
Let $\varphi\in Z^2(V,\mathfrak{so}_d(\C))^{\ms}.$ Then $(x.\varphi)(a,b)=0$ for $x\in\Sl_2(\C), \ a,b\in V.$ By the invariance of $\varphi$ with respect to $D\in\Sl(2,\C),$ we have
\begin{align}\label{Dphi}
\nonumber
(D.\varphi)(P_{i,j},P_{r,t})&=[D,\varphi(P_{i,j},P_{r,t})]-\varphi([D,P_{i,j}],P_{r,t})-\varphi(P_{i,j},[D,P_{r,t}])\\
                            &=[D,\varphi(P_{i,j},P_{r,t})]-2(2l-i-r)\varphi(P_{i,j},P_{r,t})=0.
\end{align}
From this equation we obtain that $\varphi(P_{i,j},P_{r,t})$ can survive when and only when $i+r=2l.$

{\bf Case 1.} Let $l\in\frac{1}{2}+\N_0.$ Then we fix that $i+r=2l.$ First of all, we show that $\varphi(P_{i,j},P_{r,j})=0, \ j=1,2,\dots,d.$ Set
$$\varphi(P_{i,j},P_{r,j})=\sum_{m,n}^{}\alpha_{m,n}M_{m,n}.$$
We first consider the invariance with respect to $M_{t,k}\in\so(d,\C),$ $t\neq j, \ k\neq j.$ Then we have
\begin{align*}
(M_{t,k}.\varphi)(P_{i,j},P_{r,j})&=[M_{t,k},\varphi(P_{i,j},P_{r,j})]-\varphi([M_{t,k},P_{i,j}],P_{r,j})
-\varphi(P_{i,j},[M_{t,k},P_{r,j}])\\
&=\sum_{m,n}^{}\alpha_{m,n}[M_{t,k},M_{m,n}]=0.
\end{align*}
Using the simplicity of $\so(d,\C),$ for any pair $(m,n)$ there exists at least one pair $(t,k)$ such that the commutator $[M_{t,k},M_{m,n}]$ does not vanish. Letting the indices run through $\{1,2,\dots,d\},$ we obtain that
$$\varphi(P_{i,j},P_{r,j})=0, \quad 1\leq j\leq d.$$

To prove the remaining part, we first show that $\varphi(P_{i,1},P_{r,2})=0.$ Then we demonstrate that for all $j,k\in\{1,2,\dots,d\},$ $\varphi(P_{i,j},P_{r,k})$ is in the orbit of the $\varphi(P_{i,2},P_{r,2})$ under the action of $\so(d,\C).$
The proof of $\varphi(P_{i,1},P_{r,2})=0$ is divided into three cases, since the proof of general case differs from the cases $d=3$ and $d=4.$
\begin{itemize}
\item Let $d=3.$ Considering the invariance of $\varphi$ with respect to $M_{2,3}\in\mathfrak{so}(3,\C)$ we have
    \begin{align*}
    (M_{2,3}.\varphi)(P_{i,1},P_{r,3})&=[M_{2,3},\varphi(P_{i,1},P_{r,3})]-\varphi([M_{2,3},P_{i,1}],P_{r,3})-\varphi(P_{i,1},[M_{2,3},P_{r,3}])\\
    &=[M_{2,3},\varphi(P_{i,1},P_{r,3})]-\varphi(P_{i,1},P_{r,2})=0.
    \end{align*}
    Applying the same arguments on the elements $M_{1,3}\in\mathfrak{so}_3$ and $\varphi(P_{i,3},P_{r,2}),$ we obtain
    $$\varphi(P_{i,1},P_{r,2})=[M_{1,3},\varphi(P_{i,3},P_{r,2})].$$
This together with the previous relation give us $\varphi(P_{i,1},P_{r,2})\in\langle M_{1,2}\rangle.$  On the other hand, we have $\varphi\in Z^2(V,\so(3,\C))$ which follows
    \begin{align*}
    d^2\varphi(P_{i,1},P_{i,2},P_{r,2})&=[P_{i,1},\varphi(P_{i,2},P_{r,2})]+[P_{i,2},\varphi(P_{r,2},P_{i,1})]+[P_{r,2},\varphi(P_{i,1},P_{i,2})]\\
    &+\varphi(P_{i,1},[P_{i,2},P_{r,2}])+\varphi(P_{i,2},[P_{r,2},P_{i,1}])+\varphi(P_{r,2},[P_{i,1},P_{i,2}])\\
    &=[P_{i,2},\varphi(P_{r,2},P_{i,1})]=0.
    \end{align*}
Since $[P_{i,2},M_{1,2}]\neq0,$ we get that $\varphi(P_{i,1},P_{r,2})=0.$

\item Let $d=4.$ In a similar way with the previous case we consider the invariance of $\varphi$ with respect to the elements $M_{2,3},M_{1,4},M_{2,4}\in\so(d,\C)$ and obtain that $\varphi(P_{i,1},P_{r,2})\in\langle M_{1,2},M_{3,4}\rangle$ and
 $$\varphi(P_{i,1},P_{r,2})=[M_{2,3},\varphi(P_{i,1},P_{r,3})].$$
 Similarly, one can show that $\varphi(P_{i,1},P_{r,3})\in\langle M_{1,3},M_{2,4}\rangle.$ Next, we consider the relation $d^2\varphi(P_{r,3},P_{i,1},P_{r,2})=0.$ Then, applying the substitution $\varphi(P_{i,1},P_{r,2})=[M_{2,3},\varphi(P_{i,1},P_{r,3})],$ we have
 \begin{align*}
    &[P_{r,3},[M_{2,3},\varphi(P_{i,1},P_{r,3})]]+[P_{r,2},\varphi(P_{r,3},P_{i,1})]=2[P_{r,2},\varphi(P_{r,3},P_{i,1})]=0.
 \end{align*}
The last equality relies on the Jacobi identity and the relation $[M_{2,3},[P_{r,3},\varphi(P_{i,1},P_{r,3})]]=0.$ This implies that $\varphi(P_{i,1},P_{r,3})\in\langle M_{1,3}\rangle$ and $\varphi(P_{i,1},P_{r,2})\in\langle M_{1,2}\rangle.$
Now, considering the cocycle condition for the elements $P_{i,1},P_{i,2},P_{r,2}\in V,$ we have
 \begin{align*}
    &[P_{i,1},\varphi(P_{i,2},P_{r,2})]+[P_{i,2},\varphi(P_{r,2},P_{i,1})]+[P_{r,2},\varphi(P_{i,1},P_{i,2})]\\
    &+\varphi(P_{i,1},[P_{i,2},P_{r,2}])+\varphi(P_{i,2},[P_{r,2},P_{i,1}])+\varphi(P_{r,2},[P_{i,1},P_{i,2}])=[P_{i,2},\varphi(P_{r,2},P_{i,1})]=0.
 \end{align*}
Since $[P_{i,2},M_{1,2}]\neq0,$ we get that $\varphi(P_{i,1},P_{r,2})=0.$

\item Let $d\geq5.$ Assume that
$$\varphi(P_{i,1},P_{r,2})=\sum_{m,n}^{}\alpha_{m,n}P_{m,n}.$$
Then we consider the invariance of $\varphi$ with respect to $M_{t,k}\in\mathfrak{so}(3,\C),$ $t\neq1,2, \ k\neq1,2.$
\begin{align*}
(M_{t,k}.\varphi)(P_{i,1},P_{r,2})&=[M_{t,k},\varphi(P_{i,1},P_{r,2})]-\varphi([M_{t,k},P_{i,1}],P_{r,2})
-\varphi(P_{i,1},[M_{t,k},P_{r,2}])\\
&=\sum_{m,n}^{}\alpha_{m,n}[M_{t,k},M_{m,n}]=0.
\end{align*}
By the simplicity of $\so(d,\C),$ for any pair $(m,n)$ there exists at least one pair $(t,k), \ t\neq1,2, \ k\neq1,2$ such that the commutator $[M_{t,k},M_{m,n}]$ does not vanish. Letting the indices run through $\{1,2,\dots,d\},$ we obtain that
$$\varphi(P_{i,1},P_{r,2})\in\langle M_{1,2}\rangle.$$
The remaining part of this case can be proved applying the same arguments as the last part of the previous case by considering the cocycle condition for the elements $P_{i,1},P_{i,2},P_{r,2}\in V.$ Thus, $\varphi(P_{i,1},P_{r,2})=0.$
\end{itemize}

Now, we show that the elements $\varphi(P_{i,m},P_{r,n}), \ m\neq n$ are in the orbit of $\varphi(P_{i,1},P_{r,2})$ with respect to the $\so(d,\C)$-action. First, we consider the invariance of $\varphi$ with respect to $M_{m,1},M_{n,2}\in\so(d,\C).$
\begin{align*}
(M_{m,1}.\varphi)(P_{i,1},P_{r,2})&=[M_{m,1},\varphi(P_{i,1},P_{r,2})]-\varphi([M_{m,1},P_{i,1}],P_{r,2})-\varphi(P_{i,1},[M_{m,1},P_{r,2}])\\
&=-\varphi(P_{i,m},P_{r,2})=0.
\end{align*}
Similarly, by $(M_{n,2}.\varphi)(P_{i,m},P_{r,2})=0$ we have
\begin{align*}
&[M_{n,2},\varphi(P_{i,m},P_{r,2})]-\varphi([M_{n,2},P_{i,m}],P_{r,2})-\varphi(P_{i,m},[M_{n,2},P_{r,2}])=\varphi(P_{i,m},P_{r,n})=0.
\end{align*}
Hence, we obtain that $\varphi(P_{i,m},P_{r,n})=0,$ $m,n\in\{1,2,\dots,d\}.$ Thus, $Z^2(V,\mathfrak{so}_d(\C))^{\ms}=0.$

{\bf Case 2.} Let $l\in\N_0.$ Then applying the same arguments as the previous case, we obtain that
$\varphi(P_{i,m},P_{r,n})=0,$
where $\ i\neq r, \ i+r=2l, \  m,n\in\{1,2,\dots,d\}.$ Consider the case $i=r=l.$ Set
$$\varphi(P_{l,m},P_{l,n})=\sum_{p,q}^{}\alpha_{p,q}M_{p,q},$$
Then by applying the invariance condition $(M_{t,k}.\varphi)(P_{l,m},P_{r,n})=0, \ t\neq m,n, \ k\neq m,n,$ we have
\begin{align*}
&[M_{t,k},\varphi(P_{l,m},P_{l,n})]-\varphi([M_{t,k},P_{l,m}],P_{l,n})-\varphi(P_{l,m},[M_{t,k},P_{l,n}])
=\sum_{p,q}^{}\alpha_{p,q}[M_{t,k},M_{p,q}]=0.
\end{align*}
Again by the simplicity of $\so(d,\C),$ we derive that $\varphi(P_{l,m},P_{l,n})\in\langle M_{m,n} \rangle.$ It only remains to show that $Z^2(V,\so(d,\C))^{\ms}$ is one-dimensional. For this, set
$$\varphi(P_{l,m},P_{l,n})= M_{m,n}.$$
We consider the invariance of $\varphi$ with respect to the elements $M_{i,m},M_{j,n}\in\so(d,\C), \ i\neq n, \ j\neq m.$ So, the relation $(M_{i,m}.\varphi)(P_{l,m},P_{r,n})=0$ implies that
\begin{align*}
&[M_{i,m},\varphi(P_{l,m},P_{l,n})]-\varphi([M_{i,m},P_{l,m}],P_{l,n})-\varphi(P_{l,m},[M_{i,m},P_{l,n}])\\
        &=[M_{i,m},M_{m,n}]-\varphi(P_{l,i},P_{l,n})=0.
\end{align*}
Thus, $\varphi(P_{l,i},P_{l,n})= M_{i,n}.$ Similarly, the relation $(M_{j,n}.\varphi)(P_{l,i},P_{l,n})=0$ follows that
\begin{align}\label{dim1}
\varphi(P_{l,i},P_{l,j})= M_{i,j}.
\end{align}
Hence, $Z^2(V,\so(d,\C))^{\ms}$ is $1$-dimensional and it is generated by 2-cocycles defined as \eqref{dim1}.
\end{proof}

\begin{lem}{\label{mainlem3}}
Let $d\geq3.$ Then
\begin{align*}
Z^2(V,\ms)^{\ms}&=0, \quad \mbox{ if } \ l\in\frac{1}{2}+\N_0,\\
Z^2(V,\ms)^{\ms}&\cong\C, \quad \mbox{ if } \ l\in\N_0.
\end{align*}
\end{lem}

\begin{proof}
Take $\varphi\in Z^2(V,\ms)^{\ms}.$ By the invariance of $\varphi$ with respect to $D\in\Sl(2,\C),$ we have
\begin{align}\label{Pphi}
\nonumber
(D.\varphi)(P_{i,j},P_{r,t})&=[D,\varphi(P_{i,j},P_{r,t})]-\varphi([D,P_{i,j}],P_{r,t})-\varphi(P_{i,j},[D,P_{r,t}])\\
                            &=[D,\varphi(P_{i,j},P_{r,t})]-2(2l-i-r)\varphi(P_{i,j},P_{r,t})=0.
\end{align}

Thus, $\varphi(P_{i,j},P_{r,t})$ is an eigenvector of ${\rm ad}_{D}$ with eigenvalue $2(2l-i-r).$ Therefore, $\varphi(P_{i,j},P_{r,t})$ may not vanish in one of the following possible cases:
\begin{align*}
\varphi(P_{i,j},P_{r,t})&\in \langle H\rangle, && \mbox{ if } \ i+r=2l-1,\\
\varphi(P_{i,j},P_{r,t})&\in \langle D\rangle\oplus\so(d,\C), && \mbox{ if } \ i+r=2l,\\
\varphi(P_{i,j},P_{r,t})&\in \langle C\rangle, && \mbox{ if } \ i+r=2l+1.
\end{align*}
We now consider the invariance with respect to the elements $M_{m,n}\in\so(d,\C).$ Let $i+r=2l-1.$ Then by $M_{m,n}.\varphi=0,$ we have
\begin{align*}
(M_{m,n}.\varphi)(P_{i,j},P_{r,t})&=[M_{m,n},\varphi(P_{i,j},P_{r,t})]-\varphi([M_{m,n},P_{i,j}],P_{r,t})-\varphi(P_{i,j},[M_{m,n},P_{r,t}])\\
        &=\delta_{m,j}\varphi(P_{i,n},P_{r,t})-\delta_{n,j}\varphi(P_{i,m},P_{r,t})+\delta_{m,t}\varphi(P_{i,j},P_{r,n})-\delta_{n,t}\varphi(P_{i,j},P_{r,m})=0.
\end{align*}
If $j\neq t$ and $n\neq j,t,$ then we obtain that $\varphi(P_{i,n},P_{r,t})=0.$ Letting the index $n$ run through $\{1,2,\dots,d\},$ we get that $\varphi(P_{i,j},P_{r,t})=0$ for $j\neq t.$ If $m=j, \ n=t,$ then $\varphi(P_{i,j},P_{r,j})=\varphi(P_{i,t},P_{r,t}).$ We now apply the 2-cochain condition for the elements $P_{i,j},P_{r,j},P_{r,t}\in V.$ Thus, we have
\begin{align*}
d^2\varphi(P_{i,j},P_{r,j},P_{r,t})&=[P_{i,j},\varphi(P_{r,j},P_{r,t})]+[P_{r,j},\varphi(P_{r,t},P_{i,j})]+[P_{r,t},\varphi(P_{i,j},P_{r,j})]\\
&+\varphi(P_{i,j},[P_{r,j},P_{r,t}])+\varphi(P_{r,j},[P_{r,t},P_{i,j}])+\varphi(P_{r,t},[P_{i,j},P_{r,j}])\\
&=[P_{r,t},\varphi(P_{i,j},P_{r,j})]=0.
\end{align*}
Note that $[P_{r,t},H]$ does not vanish, and thus, we obtain that $\varphi(P_{i,j},P_{r,j})=0.$ Hence,
$$\varphi(P_{i,j},P_{r,t})=0, \quad i+r=2l-1,  \quad j,t\in\{1,2,\dots,d\}.$$
Applying similar arguments, we can get that $\varphi(P_{i,j},P_{r,t})=0,$ $i+r=2l+1.$ Now, let $i+r=2l.$ Then by the invariance with respect to $C\in\Sl(2,\C),$ we have
\begin{align*}
(C.\varphi)(P_{i,j},P_{r,t})&=[C,\varphi(P_{i,j},P_{r,t})]-\varphi([C,P_{i,j}],P_{r,t})-\varphi(P_{i,j},[C,P_{r,t}])\\
        &=[C,\varphi(P_{i,j},P_{r,t})]-(2l-i)\varphi(P_{i+1,j},P_{r,t})-(2l-r)\varphi(P_{i,j},P_{r+1,t})=0.
\end{align*}
Summarizing all the relations obtained above and taking $\varphi(P_{i+1,j},P_{r,t})=0,$ $\varphi(P_{i,j},P_{r+1,t})=0$ and $[C,D]=2C$ into account, we conclude that $\varphi\in Z^2(V,\so(d,\C))^{\ms}.$ Then by Lemma \ref{mainlem2}, we get that $ Z^2(V,\ms)^{\ms}$ is either zero or one-dimensional depending on whether $l\in\frac{1}{2}+\N_0$ or $l\in\N_0.$
\end{proof}

\begin{lem}{\label{mainlem4}}
$Z^2(V,V)^{\ms}=0.$
\end{lem}

\begin{proof}
At first, note that by \eqref{Pphi}, $\varphi(P_{i,j},P_{r,t})$ is an eigenvector of ${\rm ad}_{D}$ with eigenvalue $2(2l-i-r).$ So, $\varphi(P_{i,j},P_{r,t})$ may not vanish, if it belongs to the following subspace:
\begin{align*}
\varphi(P_{i,j},P_{r,t})&\in \langle P_{i+r-l,k} \ | \ k=1,2,\dots,d \rangle.
\end{align*}
Here, we can immediately obtain that, if $l\in\frac{1}{2}+\N_0,$ then $\varphi(P_{i,j},P_{r,t})=0$ for all $P_{i,j},P_{r,t}\in V,$ that is, $ Z^2(V,V)^{\ms}=0.$

Now, let $l\in\N_0.$ Then we consider the invariance of $\varphi$ with respect to $M_{m,n}\in\so(d,\C), \ m\neq j,t, \ n\neq j,t.$ Hence, we have
\begin{align*}
(M_{m,n}.\varphi)(P_{i,j},P_{r,t})&=[M_{m,n},\varphi(P_{i,j},P_{r,t})]-\varphi([M_{m,n},P_{i,j}],P_{r,t})-\varphi(P_{i,j},[M_{m,n},P_{r,t}])\\
&=[M_{m,n},\varphi(P_{i,j},P_{r,t})]=0.
\end{align*}
Thus, we get $\varphi(P_{i,j},P_{r,t})\in\langle P_{i+r-l,j},P_{i+r-l,t}\rangle.$ Moreover, using the invariance relation $(M_{m,j}.\varphi)(P_{i,j},P_{r,t})=0, \ m\neq t,$ we have
\begin{align*}
[M_{m,j},\varphi(P_{i,j},P_{r,t})]&-\varphi([M_{m,j},P_{i,j}],P_{r,t})-\varphi(P_{i,j},[M_{m,j},P_{r,t}])\\
&=[M_{m,j},\varphi(P_{i,j},P_{r,t})]-\varphi(P_{i,m},P_{r,t})=0.
\end{align*}
This follows that $\varphi(P_{i,m},P_{r,t})\in\langle P_{i+r-l,m}\rangle.$ However, the relation $(M_{t,k}.\varphi)(P_{i,m},P_{r,k})=0$ implies that $\varphi(P_{i,m},P_{r,t})\in\langle P_{i+r-l,t}\rangle.$ Hence,
\begin{align}\label{PhiVV}
\varphi(P_{i,j},P_{r,t})=0, \quad j\neq t.
\end{align}
Let $\varphi(P_{i,j},P_{r,j})=\alpha P_{i+r-l,j}.$ Then the invariance relation $(M_{n,j}.\varphi)(P_{i,j},P_{r,j})=0$ leads to
\begin{align*}
&[M_{n,j},\varphi(P_{i,j},P_{r,j})]-\varphi([M_{n,j},P_{i,j}],P_{r,j})-\varphi(P_{i,j},[M_{n,j},P_{r,j}])\\
&=[M_{n,j},\varphi(P_{i,j},P_{r,j})]-\varphi(P_{i,n},P_{r,j})-\varphi(P_{i,j},P_{r,n})=[M_{n,j},\varphi(P_{i,j},P_{r,j})]=0.
\end{align*}
Taking $\varphi(P_{i,j},P_{r,j})\in\langle P_{i+r-l,j}\rangle$ and $[M_{n,j},P_{i+r-l,j}]\neq0$ into account, we obtain that $\varphi(P_{i,j},P_{r,j})=0.$ Hence, $ Z^2(V,V)^{\ms}=0.$
\end{proof}

Here, we note that $H^3(V,V)^\ms\cong {\rm Hom}_\ms(\Lambda^3(V),V)$ (see \cite{BU75}). In our case, $V=V_{2l\omega,\omega_1}$ and $\Lambda^3(V)$ is $C_{n}^3$-dimensional $\ms$-module, where $n=d(2l+1).$ Then, we claim that $\Lambda^3(V)$ does not contain a submodule isomorphic to $V.$ The reason for this is that the linearly independent elements in $V$ with eigenvalue $2l$ with respect to ${\rm ad}_D$ are $P_{0,1},P_{0,2},\dots,P_{0,d}.$ Since $D$ belongs to the Cartan subalgebra of $\ms,$ the maximal vector must belong to the linear span of the above mentioned vectors. However, the number of linearly independent vectors in $\Lambda^3(V)$ with eigenvalue $2l$ is more than $d.$ Indeed, the following elements are among them:
\begin{align*}
P_{2l,1}&\wedge P_{0,2}\wedge P_{0,1},\quad P_{2l,1}\wedge P_{0,2}\wedge P_{0,3}, \dots,P_{2l,1}\wedge P_{0,2}\wedge P_{0,d},\\
P_{2l,1}&\wedge P_{0,3}\wedge P_{0,1},\quad P_{2l,1}\wedge P_{0,3}\wedge P_{0,2}, \dots,P_{2l,1}\wedge P_{0,3}\wedge P_{0,d},\dots
\end{align*}
Moreover, the vector subspace $W$ of $\Lambda^3(V),$ which is the linear span of these vectors, is stable under the action of the subalgebra $\so(d,\C)$ of $\ms.$ This means that $W$ is a subspace of some irreducible submodule of $\Lambda^3(V)$ which is bigger than $V.$ If we apply generalized Schur's lemma, then we obtain that $H^3(V,V)^\ms\cong {\rm Hom}_\ms(\Lambda^3(V),V)=0.$
Now, we are ready to present the proof of the main result.
\begin{proof}[{\bf Proof of Theorem \ref{mainthm}}]
By Hochschild-Serre factorization theorem, we have $H^2(\mg,\mg)\cong H^2(V,\mg)^{\ms}.$ Moreover, by \eqref{iso23}, we have $H^2(V,V)^{\ms}=0$ and $H^2(V,\mg/V)\cong H^2(V,\ms)$ is zero or $1$-dimensional depending on whether $l\in\frac{1}{2}+\N_0$ or $l\in\N_0.$ Hence, if $l\in\frac{1}{2}+\N_0,$ then by Lemma \ref{2.1}, we get $H^2(V,\mg)^\ms=0.$ And if $l\in\N_0,$ then Lemma \ref{2.1} yields to the following exact sequence:
$$0\to H^2(V,\mg)^\ms\to \C\to 0,$$
which implies that $H^2(V,\mg)^\ms\cong\C.$ Thus, $H^2(\mg,\mg)\cong\C.$
\end{proof}

\subsection{The case $d=1$ and $l\in\N_0$}
In this case, the algebra $\mathfrak{cga}_\ell(1,\C)$ is nothing, but, semidirect product of the special linear Lie algebra $\Sl(2,\C)$ and its $(2l+1)$-dimensional irreducible module $V_{2l}.$ In \cite{Rauch}, the following result have been proved:
\begin{thm}{\label{Rauch}} Let $\mg=\Sl(2,\C)\ltimes V_{2l}.$ Then
\begin{align*}
H^2(\mg,\mg)&\cong \C, \ \mbox{ if } \ l\equiv1\mod4 \mbox{ or } \ l=2,\\
H^2(\mg,\mg)&=0, \ \mbox{ otherwise}.
\end{align*}
\end{thm}
Note that in \cite{Rauch}, the result is originally formulated in terms of the dimension of $\mg$. However, in the above theorem, we have rephrased their theorem using the parameter $l$, in order to emphasize its relevance to the conformal Galilei algebra $\mathfrak{cga}_\ell(d,\C)$ and to highlight how the result depends on this parameter.

\subsection{The case $d=2$ and $l\in\N_0$}

Let $l\in\frac{1}{2}\N, \ l\geq 1.$ For $i=0,1,\dots,2l,$ denote $P_{i,1},P_{i,2}$ by $P_i$ and $Q_i,$ respectively. Then $\mathfrak{cga}_\ell(2,\C)$ is isomorphic to the following non-trivial semidirect product:
$$\ms\ltimes \mr,$$
where $\ms=\Sl(2,\C),$ $\mr$ is $(4l+3)$-dimensional solvable Lie algebra with the basis $\{M_{1,2},P_{0},Q_{0},\dots,P_{2l},Q_{2l}\}$ and the following non-trivial commutation relations:
\begin{align*}
[M_{1,2},P_{i}]=-Q_{i}, \quad [M_{1,2},Q_{i}]=P_{i}, \quad i=0,1,\dots,2l.
\end{align*}
It should be noted that $\mr$ is decomposed as a direct sum of the submodules as follows:
\begin{align*}
V=V_1\oplus V_2\oplus \langle M_{1,2}\rangle,
\end{align*}
where $V_1=\langle P_{0},P_1,\dots,P_{2l} \rangle,$ $V_2=\langle Q_{0},Q_1,\dots,Q_{2l} \rangle.$ Then we can state the following theorem:

\begin{thm}{\label{mainthm2}} Let $\mg=\mathfrak{cga}_\ell(2,\C)$ and $l\in\frac{1}{2}\N_0.$ Then
\begin{align*}
\dim H^2(\mg,\mg)&=2, \quad \mbox{ if } \ l=0,\frac{1}{2}, 1,\\
\dim H^2(\mg,\mg)&=1,  \qquad \mbox{ otherwise. }
\end{align*}
\end{thm}

\begin{lem}\label{d2B2}
$\dim(B^2(V,\mathfrak{cga}_\ell(2,\C))^{\ms})=3.$
\end{lem}

\begin{proof}
Let $\sigma$ be an $\ms$-invariant 1-cochain. First, we show that $\sigma(P_{i}),\sigma(Q_{i})\in V, \ i\in\{0,1,\dots,2l\}, \ j=1,2.$ The invariance with respect to $D\in\ms$ leads to
\begin{align*}
(D.\sigma)(P_{i})=[D,\sigma(P_{i})]-\sigma([D,P_{i}])=[D,\sigma(P_{i})]-2(l-i)\sigma(P_{i})=0.
\end{align*}
So, $\sigma(P_{i})$ and $\sigma(Q_{i})$ are eigenvectors of ${\rm ad}_D$ with eigenvalue $2(l-i).$ Thus, we have the following relations:
\begin{align*}
\sigma(P_{l-1})&\in\langle H,P_{l-1},Q_{l-1} \rangle, &  \sigma(P_{l+1})&\in\langle C,P_{l+1},Q_{l+1}\rangle,\\
\sigma(P_{l})&\in\langle M_{1,2},P_{l},Q_{l}  \rangle, & \sigma(P_{i})&\in V, \quad i\neq l-1,l,l+1.
\end{align*}
Moreover, the invariance relations $H.\sigma=0$ and $C.\sigma=0$ imply that
\begin{align*}
(H.\sigma)(P_{l})&=[H,\sigma(P_{l})]-\sigma([H,P_{l}])=[H,\sigma(P_{l})]+l\sigma(P_{l-1})=0.
\end{align*}
Hence, $\sigma(P_{l-1})\in\langle P_{l-1}\rangle.$ Similarly, the invariance with respect to $C\in\ms$ follows that $\sigma(P_{l+1})\in\langle P_{l+1} \rangle$ and $\sigma(P_{l})\in\langle P_{l} \rangle.$ Similarly, we show that $\sigma(Q_{i})\in V.$ Thus, we have $\sigma:V\to V$ and $x.\sigma=0$ for $x\in\ms,$ that is,
$$(x.\sigma)(v)=[x,\sigma(v)]-\sigma([x,v])=0, \quad v\in V.$$
This means that $\sigma$ is an $\ms$-module homomorphism. Then by the generalized Schur's lemma we have
\begin{align}\label{InvB11}
\sigma(P_i)&=\lambda_{1,1}P_i+\lambda_{1,2}Q_i, & \sigma(Q_i)&=\lambda_{2,1}P_i+\lambda_{2,2}Q_i, &
\sigma(M_{1,2})&=\gamma M_{1,2},
\end{align}
where $\lambda_{ij},\gamma\in\C.$
Now, we define the dimension of the space of $\ms$-invariant derivations. Let $d$ be a derivation of the form \eqref{InvB11}, that is, $d(P_{i})=\alpha_{1,1}P_{i}+\alpha_{1,2}Q_{i},$ $d(Q_{i})=\alpha_{2,1}P_{i}+\alpha_{2,2}Q_{i},$ $d(M_{1,2})=\beta M_{1,2}$. Then, taking $[M_{1,2},P_{i}]=-Q_{i}$ into account, we have
\begin{align*}
-d(Q_{i})&=d([M_{1,2},P_{i}])=[d(M_{1,2}),P_{i}]+[M_{1,2},d(P_{i})]\\
&=\beta[M_{1,2},P_{i}]+\alpha_{1,1}[M_{1,2},P_{i}]+\alpha_{1,2}[M_{1,2},Q_{i}].
\end{align*}
This implies that $\beta+\alpha_{1,1}=\alpha_{2,2}, \ \alpha_{1,2}=-\alpha_{2,1}.$ Similarly, by $[M_{1,2},Q_{i}]=P_{i},$ we get that $\alpha_{11}=\alpha_{22}+\beta.$ Summarizing obtained relations, we get that $\alpha_{1,1}=\alpha_{2,2}, \ \beta=0.$
Thus, $\dim({\rm Der}(\mathfrak{cga}_\ell(2,\C))^{\ms})=2$ and hence, $\dim(B^2(V,\mathfrak{cga}_\ell(d,\C))^{\ms})=3.$
\end{proof}

\begin{lem}\label{deq2}
Let $l\geq3.$ Then $\dim(Z^2(V,\mathfrak{cga}_\ell(2,\C))^{\ms})=4.$
\end{lem}

\begin{proof}
Since the computations are different, the proof is divided into two cases depending on whether $l$ is integer or half integer.

{\bf Case 1.} Let $l\in\frac{1}{2}+\N$ and $\varphi\in Z^2(V,\mathfrak{cga}_\ell(2,\C))^{\ms}.$ Then by \eqref{Dphi}, $\varphi(P_{i},P_{j}),$ $\varphi(P_{i},Q_{j})$ and $\varphi(Q_{i},Q_{j})$ are eigenvectors of ${\rm ad}_{D}$ with eigenvalue $2(2l-i-j).$ Since $l\in\frac{1}{2}+\N,$ there is no element in $V_1\oplus V_2$ with eigenvalue $2(2l-i-j).$ Furthermore, the invariance condition follows that $\varphi(P_{i},M_{1,2})$ and $\varphi(Q_{i},M_{1,2})$ are eigenvectors of ${\rm ad}_{D}$ with eigenvalue $2(l-i).$ Thus,
\begin{align*}
\varphi(P_{i},P_{j}),\varphi(Q_{i},Q_{j}),\varphi(P_{i},Q_{j})&\in\langle M_{1,2},D,H,C\rangle,\\
\varphi(P_{i},M_{1,2}),\varphi(Q_{i},M_{1,2})&\in \langle P_{i}, Q_{i}\rangle.
\end{align*}
Therefore, we set
\begin{align*}
\varphi(P_{i},P_{j})&= \delta_{2l,i+j}a_{i,j}M_{1,2}+\delta_{2l,i+j}b_{i,j}D+\delta_{2l-1,i+j}c_{i,j}H+\delta_{2l,i+j}d_{i,j}C,\\
\varphi(P_{i},Q_{j})&= \delta_{2l,i+j}e_{i,j}M_{1,2}+\delta_{2l,i+j}f_{i,j}D+\delta_{2l-1,i+j}g_{i,j}H+\delta_{2l,i+j}h_{i,j}C,\\
\varphi(Q_{i},Q_{j})&= \delta_{2l,i+j}p_{i,j}M_{1,2}+\delta_{2l,i+j}q_{i,j}D+\delta_{2l-1,i+j}r_{i,j}H+\delta_{2l,i+j}s_{i,j}C,\\
\varphi(P_{i},M_{1,2})&= \alpha_{i}P_{i}+\beta_{i}Q_{i} \qquad \varphi(Q_{i},M_{1,2})= \gamma_{i}P_{i}+\mu_{i}Q_{i}.
\end{align*}
By the invariance of $\varphi$ with respect to $C\in\ms,$ we have
\begin{align*}
(C.\varphi)(P_{i},M_{1,2})&=[C,\varphi(P_{i},M_{1,2})]-\varphi([C,P_{i}],M_{1,2})-\varphi(P_{i},[C,M_{1,2}])\\
        &=[C,\varphi(P_{i},M_{1,2})]-(2l-i)\varphi(P_{i+1},M_{1,2})=0.
\end{align*}
This implies that $\alpha_{i}=\alpha_{i+1}, \ \beta_i=\beta_{i+1}.$ Similarly, one can obtain that $\gamma_{i}=\gamma_{i+1}, \ \mu_i=\mu_{i+1}.$
We now consider the $2$-cocycle condition $d^2\varphi=0$ for the elements $P_{i},P_{j},P_k\in V.$ Then we have
\begin{align*}
[P_{i},\varphi(P_{j},P_{k})]&+[P_{j},\varphi(P_{k},P_{i})]+[P_{k},\varphi(P_{i},P_{j})]+\varphi(P_{i},[P_{j},P_{k}])+\varphi(P_{j},[P_{k},P_{i}])+\varphi(P_{k},[P_{i},P_{j}])\\
&=\delta_{2l,j+k}a_{j,k}[P_i,M_{1,2}]+\delta_{2l,j+k}b_{j,k}[P_i,D]+\delta_{2l-1,j+k}c_{j,k}[P_i,H]+\delta_{2l+1,j+k}d_{j,k}[P_i,C]\\
&+\delta_{2l,k+i}a_{k,i}[P_j,M_{1,2}]+\delta_{2l,k+i}b_{k,i}[P_j,D]+\delta_{2l-1,k+i}c_{k,i}[P_j,H]+\delta_{2l+1,k+i}d_{k,i}[P_j,C]\\
&+\delta_{2l,i+j}a_{i,j}[P_k,M_{1,2}]+\delta_{2l,i+j}b_{i,j}[P_k,D]+\delta_{2l-1,i+j}c_{i,j}[P_k,H]+\delta_{2l+1,i+j}d_{i,j}[P_k,C]=0.
\end{align*}
Here, we can choose the indices $i,j,k$ which satisfies the condition:
\begin{align}\label{idx}
\delta_{2l,i+j}=1, \quad \delta_{2l,i+k}=\delta_{2l,j+k}=\delta_{2l-1,i+k}=\delta_{2l-1,j+k}=\delta_{2l+1,i+k}=\delta_{2l+1,j+k}=0.
\end{align}
Thus, if $i+j=2l,$ then we get that $a_{i,j}=b_{i,j}=0.$ And, if $i+j=2l-1,$ then we obtain that $c_{i,j}=0.$ Similarly, if $i+j=2l+1,$ then $d_{i,j}=0.$ Hence, $\varphi(P_{i},P_{j})=0.$ Applying the same arguments for the elements $Q_{i},Q_{j},Q_k\in V,$ we obtain that $\varphi(Q_{i},Q_{j})=0.$
Now, we consider the relation $(d^2\varphi)(P_{i},Q_{j},Q_{k})=0.$ Then we have
\begin{align*}
[P_{i},\varphi(Q_{j},Q_{k})]&+[Q_{j},\varphi(Q_{k},P_{i})]+[Q_{k},\varphi(P_{i},Q_{j})]
+\varphi(P_{i},[Q_{j},Q_{k}])+\varphi(Q_{j},[Q_{k},P_{i}])+\varphi(Q_{k},[P_{i},Q_{j}])\\
&=-\delta_{2l,k+i}e_{i,k}[Q_j,M_{1,2}]-\delta_{2l,k+i}f_{i,k}[Q_j,D]-\delta_{2l-1,k+i}g_{i,k}[Q_j,H]-\delta_{2l+1,k+i}h_{i,k}[Q_j,C]\\
&+\delta_{2l,i+j}e_{i,j}[Q_k,M_{1,2}]+\delta_{2l,i+j}f_{i,j}[Q_k,D]+\delta_{2l-1,i+j}g_{i,j}[P_k,H]+\delta_{2l+1,i+j}h_{i,j}[Q_k,C]=0.
\end{align*}
Again by choosing the indices $i,j$ to be $i+j=2l$ and $k$ with the condition \eqref{idx}, we obtain that $e_{i,j}=f_{i,j}=0.$ Similarly, we get that $g_{i,j}=h_{i,j}=0.$ Thus, $\varphi(P_i,Q_j)=0.$
Summarizing all the relations obtained above, we derive that $\dim(Z^2(V,\mathfrak{cga}_\ell(2,\C))^{\ms})=4.$

{\bf Case 2.} Let $l\in\N$ and $\varphi\in Z^2(V,\mathfrak{cga}_\ell(2,\C))^{\ms}.$ Then as in the previous case $\varphi(P_{i},P_{j})$, $\varphi(P_{i},Q_{j})$ and $\varphi(Q_{i},Q_{j})$ are eigenvector of ${\rm ad}_{D}$ with eigenvalue $2(2l-i-j).$ Since $l\in\N,$ there exist elements in $V_1\oplus V_2$ with eigenvalue $2(2l-i-j)$ with respect to ${\rm ad}_{D}.$ Furthermore, the invariance condition follows that $\varphi(P_{i},M_{1,2})$ and $\varphi(Q_{i},M_{1,2})$ is eigenvector of ${\rm ad}_{D}$ with eigenvalue $2(l-i).$
Therefore, we have the following:
\begin{align*}
\varphi(P_{i},P_{j}),\varphi(Q_{i},Q_{j}),\varphi(P_{i},Q_{j})&\in\langle P_{i+j-l},Q_{i+j-l},M_{1,2},D,H,C\rangle,\\
\varphi(P_{i},M_{1,2}),\varphi(Q_{i},M_{1,2})&\in \langle P_{i}, Q_{i}\rangle.
\end{align*}
First we set
\begin{align*}
\varphi(P_{i},P_{j})&=a_{i,j}P_{i+j-l}+b_{i,j}Q_{i+j-l}+X, \quad \varphi(Q_{i},Q_{j})=c_{i,j}P_{i+j-l}+d_{i,j}Q_{i+j-l}+Y,\\
\varphi(P_{i},Q_{j})&=e_{i,j}P_{i+j-l}+f_{i,j}Q_{i+j-l}+Z, \quad \varphi(P_{i},M_{1,2})=\alpha_{i}P_{i}+\beta_{i}Q_{i}, \quad \varphi(Q_{i},M_{1,2})=\gamma_{i}P_{i}+\mu_{i}Q_{i},
\end{align*}
where $X,Y,Z\in\langle M_{1,2},D,H,C\rangle.$ Then we first apply the invariance of $\varphi$ with respect to $C\in\ms.$ This follows the similar result as in the previous case, that is, $\alpha_{i}=\alpha_{i+1},$ $\beta_i=\beta_{i+1},$ $\gamma_{i}=\gamma_{i+1}, \ \mu_i=\mu_{i+1}.$ After that by using the $2$-cocycle identity for the same triples of elements as in the previous case, we obtain that $X=Y=Z=0.$
Now, we consider the $2$-cocycle condition for the elements $P_{i},P_{j},M_{1,2}\in V.$ Then we have
\begin{align*}
d^2(P_{i},P_{j},M_{1,2})&=[P_{i},\varphi(P_{j},M_{1,2})]+[P_{j},\varphi(M_{1,2},P_{i})]+[M_{1,2},\varphi(P_{i},P_{j})]\\
&+\varphi(P_{i},[P_{j},M_{1,2}])+\varphi(P_{j},[M_{1,2},P_{i}])+\varphi(M_{1,2},[P_{i},P_{j}])\\
&=(-a_{i,j}+f_{i,j}-f_{j,i})Q_{i+j-l}+(b_{i,j}+e_{i,j}-e_{j,i})P_{i+j-l}=0,
\end{align*}
which follows that
\begin{align}\label{111}
\nonumber
-a_{i,j}+f_{i,j}-f_{j,i}&=0,\\
b_{i,j}+e_{i,j}-e_{j,i}&=0,
\end{align}
Similarly, by $(d^2\varphi)(Q_{i},Q_{j},M_{1,2})=0$ and $(d^2\varphi)(P_{i},Q_{j},M_{1,2})=0,$ we derive the following relations, respectively:

\begin{align}\label{222}
\nonumber
-c_{i,j}-f_{i,j}+f_{j,i}&=0,\\
d_{i,j}-e_{i,j}+e_{j,i}&=0
\end{align}
and
\begin{align}\label{333}
\nonumber
-e_{i,j}-b_{i,j}+d_{j,i}&=0,\\
f_{i,j}-a_{i,j}+c_{j,i}&=0.
\end{align}

Moreover, one can directly verify that the equality $(d^2\varphi)(P_{i},Q_{i},M_{1,2})=0$ implies that $\varphi(P_i,Q_i)=0.$ Finally, summarizing the equalities \eqref{111},\eqref{222} and \eqref{333}, we get that
\begin{align*}
\varphi(P_{i},P_{j})&=\varphi(Q_{i},Q_{j})=\varphi(P_{i},Q_{j})=0, &\\
\varphi(P_{i},M_{1,2})&=uP_{i}+vQ_{i}, \qquad \varphi(Q_{i},M_{1,2})=wP_{i}+zQ_{i}.
\end{align*}
Thus, we derive that $\dim(Z^2(V,\mathfrak{cga}_\ell(2,\C))^{\ms})=4.$
\end{proof}

\begin{proof}[{\bf Proof of Theorem \ref{mainthm2}}]
Let $l=0,$ then $\mathfrak{cga}_0(2,\C)$ is isomorphic to $\Sl(2,\C)\oplus \mr,$ where $\mr$ is $3$-dimensional solvable Lie algebra with the basis $\{M_{1,2},P_{0,1},P_{0,2}\}$ and the following commutation relations:
\begin{align*}
[M_{1,2},P_{0,1}]=-P_{0,2}, \quad [M_{1,2},P_{0,2}]=P_{0,1}.
\end{align*}
Then direct computations show that $\dim(Z^2(\mathfrak{cga}_0(2,\C),\mathfrak{cga}_0(2,\C)))=31.$ Moreover, the matrix form of the derivations of $\mathfrak{cga}_0(2,\C)$ with respect to the given basis is as follows:
$$\left(
\begin{array}{cccccc}
 a_{1,1} & 0 & a_{1,3} & 0 & 0 & 0 \\
 0 & -a_{1,1} & a_{2,3} & 0 & 0 & 0 \\
 -2a_{2,3} & -2a_{1,3} & 0 & 0 & 0 & 0 \\
 0 & 0 & 0 & 0 & a_{4,5} & a_{4,6} \\
 0 & 0 & 0 & 0 & a_{5,5} & a_{5,6} \\
 0 & 0 & 0 & 0 & -a_{5,6} & a_{5,5} \\
\end{array}
\right).$$
Thus, $\dim({\rm Der}(\mathfrak{cga}_0(2,\C)))=7$ and $\dim(B^2(\mathfrak{cga}_0(2,\C),\mathfrak{cga}_0(2,\C)))=36-\dim({\rm Der}(\mathfrak{cga}_0(2,\C)))=29.$ Hence, $\dim(H^2(\mathfrak{cga}_0(2,\C),\mathfrak{cga}_0(2,\C)))=2.$
Now, if $l=\frac{1}{2},$ then $\mathfrak{cga}_{\frac{1}{2}}(2,\C)$ is nothing but, the centerless Schr\"odinger algebra $S_2.$ In \cite{Ru}, it is shown that the second adjoint cohomology space of $S_2$ is 2-dimensional. For other small values of $l,$ we have computed the adjoint cohomology groups $H^2(\mg,\mg)$ by computer and have compared with general case. The results are presented in the table below.

\begin{center}
\begin{tabular}{c|ccccc}\label{table}
$\mg$             & $\dim \mg$ & $\dim {\rm Der}(\mg)$ & $\dim B^2(\mg,\mg)$ & $\dim Z^2(\mg,\mg)$  & $\dim H^2(\mg,\mg)$ \\[4pt]
\hline
$\mathfrak{cga}_0(2,\C)$           & $6$  & $7$   & $29$ & $31$ & $2$ \\[4pt]
$\mathfrak{cga}_\frac{1}{2}(2,\C)$ & $8$  & $9$   & $54$ & $56$ & $2$ \\[4pt]
$\mathfrak{cga}_1(2,\C)$           & $10$ & $11$   & $89$ & $91$ & $2$ \\[4pt]
$\mathfrak{cga}_\frac{3}{2}(2,\C)$ & $12$ & $13$   & $131$ & $132$ & $1$ \\[4pt]
$\mathfrak{cga}_2(2,\C)$           & $14$ & $15$   & $181$ & $182$ & $1$ \\[4pt]
$\mathfrak{cga}_\frac{5}{2}(2,\C)$ & $16$ & $17$   & $239$ & $240$ & $1$
\end{tabular}
\end{center}
If $l\geq3,$ then we can always find indices $i,j,k$ satisfying \eqref{idx}. Thus, by Lemmas \ref{d2B2},\ref{deq2} and applying the Hochschild-Serre factorization theorem, we obtain the assertion.
\end{proof}

\begin{rem}
The underlying reason for the non-vanishing second adjoint cohomology of \(\mathfrak{cga}_\ell(2,\C)\) is that the rotation generator \(M_{1,2}\) is not part of the Levi subalgebra of \(\mathfrak{cga}_\ell(2,\C)\). As a result, it obstructs the possibility of expressing \(\mathfrak{sl}(2,\mathbb{C})\)-invariant cocycles as coboundaries.
\end{rem}

\section{The second adjoint cohomology space of $\widehat{\mathfrak{cga}}_\ell(d,\C)$}

In this section, we determine the second cohomology of the mass central extension of the conformal Galilei algebra $\mathfrak{cga}_\ell(d,\C).$ Namely, we show that the second cohomology space of $\widehat{\mathfrak{cga}}_\ell(d,\C)$ with values in the adjoint module vanishes. At first, we note that $\mathfrak{cga}_\ell(d,\C)$ is isomorphic to the factor algebra $\widehat{\mathfrak{cga}}_\ell(d,\C)$ by its one-dimensional center. Also, we must note that $\widehat{\mathfrak{cga}}_\ell(d,\C)$ has the following non-trivial semidirect product:
$$\widehat{\mathfrak{cga}}_\ell(d,\C)=\ms\ltimes \mn,$$
where $\ms=\Sl(2,\C)\oplus\so(d,\C),$ the radical $\mn$ is $2$-step nilpotent Lie algebra and is decomposed as the direct sum of $\ms$-submodules as follows:
$$V=V_{2l\omega,\omega_1}\oplus\C,$$
where $V_{2l\omega,\omega_1}$ is $d(2l+1)$-dimensional $\ms$-module introduced as in Section 3 and the central element $M$ generates one dimensional trivial submodule.

\subsection{The case $d\geq3$}

\begin{thm}{\label{mainthm3}} Let $\mg=\widehat{\mathfrak{cga}}_\ell(d,\C)$ and $d\geq3.$ Then
$H^2(\mg,\mg)=0.$
\end{thm}
\begin{proof}
Let $\sigma:V\to\mg$ be an $\ms$-invariant cochain. Then we apply the same arguments as in the proof of Lemma \ref{mainlem1} and show that $\sigma(V)\subseteq V.$ Moreover, the Schur's lemma we have
$$\sigma=\alpha{\rm Id}_{|V}+\theta,$$
where $\alpha\in\C$ and $\theta_{|V}\equiv 0,$ $\theta(M)=\lambda M, \ \lambda\in\C.$ However, if $\lambda=2\alpha,$ then $\sigma$ is a derivation. Thus, we get $\dim(B^{2}(V,\mg)^\ms)=1.$

Now, let $\varphi$ be an $\ms$-invariant $2$-cocycle. Define a linear mapping $\phi:V\rightarrow\mg$ as follows:
$$\phi(-)=\varphi(-,M),$$
where $M\in V$ is the central element. Then, the invariance $\varphi$ with respect to $x\in\ms$ implies that
\begin{align*}
[x,\varphi(a,M)]=\varphi([x,a],M) \ \mbox{ or } \ [x,\phi(a)]=\phi([x,a]),
\end{align*}
where $a\in V.$ In other words, $\phi$ is invariant with respect to subalgebra $\mathfrak{s}$. Using the $\ms$-invariance of $\phi$ and Schur's lemma we have $\phi=c{\rm Id}_{|V}+\delta.$ Moreover, doing the same calculations as in the proofs of Lemmas \ref{mainlem2}, \ref{mainlem3} and \ref{mainlem4} we obtain that $\varphi\in Z^2(V,V)^{\ms}.$ Since $l\in\frac{1}{2}+\N_0,$ we have
$$\varphi(P_{i,j},P_{r,t})=\delta_{i+r,2l}\alpha_{i,r}M,$$
where $P_{i,j},P_{r,t}\in V,$ that is, $ Z^2(V,V)^{\ms}=0.$ Using the invariance relation $M_{j,m}.\varphi=0, \ j\neq t,$ we have
\begin{align*}
(M_{j,m}.\varphi)(P_{i,m},P_{r,t})&=[M_{j,m},\varphi(P_{i,m},P_{r,t})]-\varphi([M_{j,m},P_{i,m}],P_{r,t})-\varphi(P_{i,m},[M_{j,m},P_{r,t}])\\
&=[M_{j,m},\varphi(P_{i,m},P_{r,t})]-\varphi(P_{i,m},P_{r,t})=0.
\end{align*}
This follows that $\varphi(P_{i,j},P_{r,t})=\delta_{2l,i+r}\delta_{j,t}\alpha_{i,r}M.$ Furthermore, the invariance of $\varphi$ with respect to $M_{j,m},$ follows that
\begin{align*}
(M_{j,t}.\varphi)(P_{i,j},P_{r,t})&=[M_{j,t},\varphi(P_{i,j},P_{r,t})]-\varphi([M_{j,t},P_{i,j}],P_{r,t})-\varphi(P_{i,j},[M_{j,t},P_{r,t}])\\
&=\varphi(P_{i,j},P_{r,j})-\varphi(P_{i,t},P_{r,t})=0.
\end{align*}
Now, by the invariance with respect to $C\in\Sl(2,\C),$ we have
\begin{align*}
(C.\varphi)(P_{i-1,j},P_{r,t})&=[C,\varphi(P_{i-1,j},P_{r,t})]-\varphi([C,P_{i-1,j}],P_{r,t})-\varphi(P_{i-1,j},[C,P_{r,t}])\\
        &=-(r+1)\varphi(P_{i,j},P_{r,t})-i\varphi(P_{i-1,j},P_{r+1,t})=0.
\end{align*}
Thus, $i\varphi(P_{i-1,j},P_{r+1,t})=-(r+1)\varphi(P_{i,j},P_{r,t}).$ Letting the indices $i,r$ run through $\{0,1,\dots,2l\},$ we obtain that
\begin{align*}
\varphi(P_{i,j},P_{r,t})&=-\frac{i}{r+1}\varphi(P_{i-1,j},P_{r+1,t})=\frac{i(i-1)}{(r+1)(r+2)}\varphi(P_{i-2,j},P_{r+2,t})\\
&=-\frac{i(i-1)(i-2)}{(r+1)(r+2)(r+3)}\varphi(P_{i-3,j},P_{r+3,t})=\dots=(-1)^i\frac{i!r!}{(2l)!}\varphi(P_{0,j},P_{2l,t}).
\end{align*}
Setting $\alpha=(-1)^{l+\frac{1}{2}}(2l)!\alpha_{0,2l},$ we get
$$\varphi(P_{i,j},P_{r,t})=\delta_{2l,i+r}\delta_{j,t}b_i\alpha M,$$
where $\delta_{j,t}$ is the Kronecker symbol and $b_i$ is defined as \eqref{3.2}.
Hence, $\dim(Z^2(V,\mg)^{\ms})=1.$ Again by the Hochschild-Serre factorization theorem, we obtain that $H^2(\mg,\mg)\cong H^2(V,\mg)^\ms=0.$
\end{proof}

\subsection{The case $d=1$}
First, note that $\widehat{\mathfrak{cga}}_\ell(1,\C)$ has the following non-trivial semidirect product:
$$\widehat{\mathfrak{cga}}_\ell(1,\C)=\ms\ltimes \mh,$$
where $\ms=\Sl(2,\C),$ the radical $\mh$ is $(2l+2)$-dimensional Heisenberg algebra and is decomposed as the direct sum of $\ms$-submodules as follows:
$$V=V_{2l}\oplus\langle M\rangle,$$
where $V_{2l}$ is $(2l+1)$-dimensional standard irreducible $\ms$-module. Denote the elements $P_{i,0}\in V_{2l}$ by $P_{i}$ for $i=0,1,\dots,2l.$
By Theorem \ref{Rauch}, we have $H^2(\mathfrak{cga}_\ell(1,\C),\mathfrak{cga}_\ell(1,\C))=0$ for $l\in\frac{1}{2}+\N_0.$ The following theorem states that the second adjoint cohomology space of mass extension of $\mathfrak{cga}_\ell(1,\C)$ vanishes.

\begin{thm}{\label{mainthm4}} Let $\mg=\widehat{\mathfrak{cga}}_\ell(1,\C).$ Then
$H^2(\mg,\mg)=0.$
\end{thm}

\begin{proof}
Let $\sigma:V\to\mg$ be an $\ms$-invariant cochain. Applying the same arguments as in the proof of Lemma \ref{mainlem1} we get that $\sigma(V)\subseteq V.$ Moreover, by the Schur's lemma we have
$$\sigma=\alpha{\rm Id}_{|V}+\delta,$$
where $\alpha\in\C$ and $\delta_{|V}\equiv 0,$ $\delta(M)=\lambda M, \ \lambda\in\C.$ However, if $\lambda=2\alpha,$ then $\sigma$ is a derivation. Thus, we get $\dim(B^{2}(V,\mg)^{\ms})=1.$
Now, let $\varphi\in Z^2(V,\mg)^{\ms}.$ By the same technique as in the proof of Theorem \ref{mainthm3}, we obtain that $\varphi\in Z^2(V,V)^{\ms}.$ Taking $l\in\frac{1}{2}+\N_0$ into account, invariant 2-cocycles with respect to $D\in\ms$ are of the following form:
$$\varphi(P_{i},P_{j})=\delta_{2l,i+j}c_{i,j}M,$$
where $c_{i,j}\in\C.$ Now, by the invariance with respect to $C\in\ms,$ we have
\begin{align*}
(C.\varphi)(P_{i-1},P_{j})&=[C,\varphi(P_{i-1},P_{j})]-\varphi([C,P_{i-1}],P_{j})-\varphi(P_{i-1},[C,P_{j}])\\
        &=-(2l-i+1)\varphi(P_{i},P_{j})-(2l-j)\varphi(P_{i-1},P_{j+1})=0.
\end{align*}
which implies that $i\varphi(P_{i-1},P_{j+1})=-(j+1)\varphi(P_{i},P_{j}).$ Letting the indices $i,j$ run through $\{0,1,2,\dots,2l\}$ we obtain that $$\varphi(P_{i},P_{j})=(-1)^i\frac{i!j!}{(2l)!}\varphi(P_{0},P_{2l}).$$
Summarizing all the relations obtained above we get that $ \dim(Z^2(V,\mg)^{\ms})=1.$ Hence, $H^2(\mg,\mg)\cong H^2(V,\mg)^\ms=0.$
\end{proof}

\subsection{The case $d=2$}
For $i=0,1,\dots,2l,$ denote $P_{i,1},P_{i,2}$ by $P_i$ and $Q_i,$ respectively. Then $\mg=\widehat{\mathfrak{cga}}_\ell(2,\C)$ is isomorphic to the following non-trivial semidirect product:
$$\ms\ltimes \mr,$$
where $\ms=\Sl(2,\C),$ $\mr$ is $(4l+4)$-dimensional solvable Lie algebra with the basis $\{M_{1,2},P_{0},Q_{0},\dots,P_{2l},Q_{2l},M\}$ and the following non-trivial commutation relations:
\begin{align*}
[M_{1,2},P_{i}]&=-Q_{i}, & [M_{1,2},Q_{i}]&=P_{i},\\
 [P_{m},P_{2l-m}]&=b_m M, & [Q_{m},Q_{2l-m}]&=b_m M,
\end{align*}
where $b_m=(-1)^{m+\ell+\frac{1}{2}}(2\ell-m)!m!,\ i,m=0,1,\dots,2l.$ It should be noted that $\mr$ is decomposed as a direct sum of the submodules as follows:
\begin{align*}
V=V_1\oplus V_2\oplus \langle M\rangle\oplus \langle M_{1,2}\rangle,
\end{align*}
where $V_1=\langle P_{0},P_1,\dots,P_{2l} \rangle,$ $V_2=\langle Q_{0},Q_1,\dots,Q_{2l} \rangle.$ If $l=\frac{1}{2},$ then $\widehat{\mathfrak{cga}}_{\frac{1}{2}}(2,\C)$ is nothing but, the Schr\"odinger algebra $\widehat{S}_2.$ In \cite{Ru}, it is shown that the second adjoint cohomology space of $\widehat{S}_2$ is 2-dimensional. For $l\geq\frac{3}{2}$ we state the following result:

\begin{thm}{\label{mainthm5}}
Let $\mg=\widehat{\mathfrak{cga}}_\ell(2,\C).$ Then we have $H^2(\mg,\mg)\cong\C.$
\end{thm}

\begin{proof}
Let $\sigma$ be an $\ms$-invariant 1-cochain. Then $\sigma(P_{i}),\sigma(Q_{i}), \ i\in\{0,1,\dots,2l\}$ and $\sigma(M_{1,2}),\sigma(M)$ are eigenvectors of ${\rm ad}_D$ with eigenvalue $2(l-i)$ and $0,$ respectively. Note that $0$ does not occur as a weight since $l\in\frac{1}{2}+\N_0$, and hence, we have the following relations:
\begin{align*}
\sigma(P_{i}),\sigma(Q_{i})&\in\langle P_{i},Q_{i} \rangle, \quad  \sigma(M_{1,2}),\sigma(M)\in\langle M_{1,2},M\rangle.
\end{align*}
Moreover, by the invariance relation $x.\sigma=0$ for $x\in\ms,$ we get that $\sigma:V\to V$ is an $\ms$-module homomorphism. Then by the generalized Schur's lemma we have
\begin{align}\label{InvB1}
\sigma(P_i)&=\lambda_{1,1}P_i+\lambda_{1,2}Q_i, & \sigma(Q_i)&=\lambda_{2,1}P_i+\lambda_{2,2}Q_i,\\
\label{InvB2}
\sigma(M_{1,2})&=\gamma_{1,1}M_{1,2}+\gamma_{1,2}M, & \sigma(M)&=\gamma_{2,1}M_{1,2}+\gamma_{2,2}M,
\end{align}
where $\lambda_{i,j},\gamma_{i,j}\in\C.$ Now, we define the dimension of the space of $\ms$-invariant derivations. Let $d$ be a derivation defined as \eqref{InvB1} and \eqref{InvB2}. Then, taking $[M_{1,2},P_{i}]=-Q_{i}$ into account, we have
\begin{align*}
-d(Q_{i})&=d([M_{1,2},P_{i}])=[d(M_{1,2}),P_{i}]+[M_{1,2},d(P_{i})]\\
&=\gamma_{1,1}[M_{1,2},P_{i}]+\lambda_{1,1}[M_{1,2},P_{i}]+\lambda_{1,2}[M_{1,2},Q_{i}].
\end{align*}
This implies that $\lambda_{2,2}=\lambda_{1,1}+\gamma_{1,1}, \ \lambda_{1,2}=-\lambda_{2,1}.$ Similarly, by $[M_{1,2},Q_{i}]=P_{i},$ we get that $\lambda_{1,1}=\lambda_{2,2}+\gamma_{1,1}.$ Hence, we obtain that $\lambda_{1,1}=\lambda_{2,2}, \ \gamma_{1,1}=0.$ Moreover, by $[P_{m},P_{2l-m}]=b_m M,$ we get that $\gamma_{2,1}=0, \ \gamma_{2,2}=2\lambda_{1,1}.$ Thus, $\dim({\rm Der}(\mg)^{\ms})=3$ and hence, $\dim(B^2(V,\mg)^{\ms})=5.$

Now, we determine the $\ms$-invariant cocycles. Let $\varphi\in Z^2(V,\mg)^{\ms}.$ Then by \eqref{Dphi}, $\varphi(P_{i},P_{j}),$ $\varphi(P_{i},Q_{j})$ and $\varphi(Q_{i},Q_{j})$ are eigenvectors of ${\rm ad}_{D}$ with eigenvalue $2(2l-i-j).$ Since $l\in\frac{1}{2}+\N,$ there is no element in $V_1\oplus V_2$ with eigenvalue $2(2l-i-j).$ Therefore, we can derive
$$\varphi(P_i,P_j),\varphi(P_i,Q_j),\varphi(Q_i,Q_j)\in\langle M_{1,2},D,H,C, M\rangle,$$
Furthermore, the invariance condition follows that
$[D,\varphi(P_{i},M_{1,2})]=2(l-i)\varphi(P_{i},M_{1,2}).$
So, $\varphi(P_{i},M_{1,2})$ eigenvector of ${\rm ad}_{D}$ with eigenvalue $2(l-i)$ so, are
$\varphi(Q_{i},M_{1,2}),$ $\varphi(P_{i},M)$ and $\varphi(Q_{i},M).$ Moreover, we can show that $\varphi(M_{1,2},M)$ is in zero eigenspace.
Therefore, we can set
\begin{align*}
\varphi(P_{i},P_{j})&= \delta_{2l,i+j}a_{i,j}M_{1,2}+\delta_{2l,i+j}b_{i,j}D+\delta_{2l-1,i+j}c_{i,j}H+\delta_{2l+1,i+j}d_{i,j}C+\delta_{2l,i+j}x_{i,j}M,\\
\varphi(P_{i},Q_{j})&= \delta_{2l,i+j}e_{i,j}M_{1,2}+\delta_{2l,i+j}f_{i,j}D+\delta_{2l-1,i+j}g_{i,j}H+\delta_{2l+1,i+j}h_{i,j}C+\delta_{2l,i+j}y_{i,j}M,\\
\varphi(Q_{i},Q_{j})&= \delta_{2l,i+j}p_{i,j}M_{1,2}+\delta_{2l,i+j}q_{i,j}D+\delta_{2l-1,i+j}r_{i,j}H+\delta_{2l+1,i+j}s_{i,j}C+\delta_{2l,i+j}z_{i,j}M,\\
\varphi(P_{i},M_{1,2})&= \alpha_{i}P_{i}+\beta_{i}Q_{i} \qquad \varphi(Q_{i},M_{1,2})= \gamma_{i}P_{i}+\mu_{i}Q_{i}, \qquad \varphi(M_{1,2},M)= \rho M_{1,2}+\varrho M,\\
\varphi(P_{i},M)&= \varepsilon_{i}P_{i}+\xi_{i}Q_{i} \qquad \varphi(Q_{i},M)= \eta_{i}P_{i}+\zeta_{i}Q_{i}.
\end{align*}
Now, we repeat the computation that is done in the proof of Lemma \ref{deq2}. We first apply the invariance of $\varphi$ with respect to $C\in\ms.$ This follows the similar result as in the previous case, that is, $\alpha_{i}=\alpha_{i+1},$ $\beta_i=\beta_{i+1},$ $\gamma_{i}=\gamma_{i+1}, \ \mu_i=\mu_{i+1}.$ Similarly, the relations $(C.\varphi)(P_i,M)=0$ and $(C.\varphi)(Q_i,M)=0$ yields to $\varepsilon_i=\varepsilon_{i+1},$ $\xi_i=\xi_{i+1},$ $\eta_i=\eta_{i+1},$ $\zeta_i=\zeta_{i+1}.$
We now consider the $2$-cocycle condition $d^2\varphi=0$ for the elements $P_{i},P_{j},P_k\in V.$ Then we have
\begin{align*}
[P_{i},\varphi(P_{j},P_{k})]&+[P_{j},\varphi(P_{k},P_{i})]+[P_{k},\varphi(P_{i},P_{j})]+\varphi(P_{i},[P_{j},P_{k}])+\varphi(P_{j},[P_{k},P_{i}])+\varphi(P_{k},[P_{i},P_{j}])\\
&=\delta_{2l,j+k}a_{j,k}[P_i,M_{1,2}]+\delta_{2l,j+k}b_{j,k}[P_i,D]+\delta_{2l-1,j+k}c_{j,k}[P_i,H]+\delta_{2l+1,j+k}d_{j,k}[P_i,C]\\
&+\delta_{2l,k+i}a_{k,i}[P_j,M_{1,2}]+\delta_{2l,k+i}b_{k,i}[P_j,D]+\delta_{2l-1,k+i}c_{k,i}[P_j,H]+\delta_{2l+1,k+i}d_{k,i}[P_j,C]\\
&+\delta_{2l,i+j}a_{i,j}[P_k,M_{1,2}]+\delta_{2l,i+j}b_{i,j}[P_k,D]+\delta_{2l-1,i+j}c_{i,j}[P_k,H]+\delta_{2l+1,i+j}d_{i,j}[P_k,C]\\
&+\delta_{2l,j+k}b_j(\varepsilon P_i+\xi Q_i)+\delta_{2l,i+k}b_k(\varepsilon P_j+\xi Q_j)+\delta_{2l,i+j}b_i(\varepsilon P_k+\xi Q_k)=0.
\end{align*}
Here, we can choose the indices $i,j,k$ which satisfies the condition \eqref{idx}. If $i+j=2l,$ then we get that $a_{i,j}=-b_i\xi, \ 2(l-k)b_{i,j}=b_i\varepsilon,$
where $b_i$ is defined as \eqref{3.2}. Since the right hand side of the last equation does not depend on $k,$ hence, we derive $b_{i,j}=0$ and $\varepsilon=0.$ Furthermore, if $i+j=2l-1,$ then we obtain that $c_{i,j}=0.$ Similarly, if $i+j=2l+1,$ then $d_{i,j}=0.$ Hence,
$$\varphi(P_{i},P_{j})=\delta_{2l,i+j}(b_i\xi M_{1,2}+x_{i,j}M), \qquad \varphi(P_{i},M)= \xi Q_{i}.$$
Applying the same arguments for the elements $Q_{i},Q_{j},Q_k\in V,$ we obtain that $\varphi(Q_{i},Q_{j})=\delta_{2l,i+j}(b_i\eta M_{1,2}+z_{i,j}M),$ $\varphi(Q_{i},M)= \eta P_{i}.$

Now, consider the relation $(d^2\varphi)(P_{i},Q_{j},Q_{k})=0.$ Then we have
\begin{align*}
&[P_{i},\varphi(Q_{j},Q_{k})]+[Q_{j},\varphi(Q_{k},P_{i})]+[Q_{k},\varphi(P_{i},Q_{j})]
+\varphi(P_{i},[Q_{j},Q_{k}])+\varphi(Q_{j},[Q_{k},P_{i}])+\varphi(Q_{k},[P_{i},Q_{j}])\\
&=\delta_{2l,j+k}b_j\eta[P_i,M_{1,2}]-\delta_{2l,k+i}(e_{i,k}[Q_j,M_{1,2}]+f_{i,k}[Q_j,D])-\delta_{2l-1,k+i}g_{i,k}[Q_j,H]-\delta_{2l,k+i}h_{i,k}[Q_j,C]\\
&+\delta_{2l,i+j}(e_{i,j}[Q_k,M_{1,2}]+f_{i,j}[Q_k,D])+\delta_{2l-1,i+j}g_{i,j}[P_k,H]+\delta_{2l+1,i+j}h_{i,j}[Q_k,C]\\
&+\delta_{2l,j+k}b_j(\varepsilon P_i+\xi Q_i)=0.
\end{align*}
Again by choosing the indices $i,j$ to be $i+j=2l$ and $k$ with the condition \eqref{idx}, we obtain that $e_{i,j}=f_{i,j}=0.$ Similarly, if $j+k=2l,$ then $\eta+\xi=0, \ \varepsilon=0.$ Moreover, the cases $i+j=2l-1$ and $i+j=2l+1$ give us $g_{i,j}=h_{i,j}=0.$ Thus, $\varphi(P_i,Q_j)=y_{i,j}M.$

Next, if we consider the invariance condition $C.\varphi=0$ for the elements $P_{i},P_{j-1}\in V, \ i+j=2l,$ and substitute the values of $\varphi,$ then we have
\begin{align*}
&[C,\varphi(P_{j},P_{j-1})]-\varphi([C,P_{i}],P_{j-1})-\varphi(P_{i},[C,P_{j-1}])=-j\varphi(P_{i+1},P_{j-1})-(i+1)\varphi(P_{i},P_{j})=0,
\end{align*}
which follows that $jx_{i+1,j-1}=-(i+1)x_{i,j}$ for $i+j=2l.$
Applying similar arguments for the elements $Q_i,Q_{j-1}\in V, \ i+j=2l,$ we derive that $jz_{i+1,j-1}=-(i+1)z_{i,j}$ for $i+j=2l.$
Moreover, the relation $(C.\varphi)(P_{i},Q_{j-1})=0,$ leads to
\begin{align*}
[C,\varphi(P_{i},Q_{j-1})]-\varphi([C,P_{i}],Q_{j-1})-\varphi(P_{i},[C,Q_{j-1}])
=-j\varphi(P_{i+1},Q_{j-1})-(i+1)\varphi(P_{i},Q_{j})=0,
\end{align*}
which follows that $jy_{i+1,j-1}=-(i+1)y_{i,j}$ for $i+j=2l.$ If we set $x_{2l,0}=(2l)!x, \ y_{2l,0}=(2l)!y, \ z_{2l,0}=(2l)!z$ and summarize the relations obtained up to now, then
\begin{align*}
\varphi(P_{i},P_{j})&= -b_{i}\xi M_{1,2}+b_{j}xM, & \varphi(P_{i},Q_{j})&= b_{j}yM, & \varphi(Q_{i},Q_{j})&= b_i\xi M_{1,2}+b_jzM,\\
\varphi(P_{i},M_{1,2})&= \alpha P_{i}+\beta Q_{i} & \varphi(Q_{i},M_{1,2})&= \gamma P_{i}+\mu Q_{i}, & \varphi(M_{1,2},M)&= \rho M_{1,2}+\varrho M,\\
\varphi(P_{i},M)&= \xi Q_{i} & \varphi(Q_{i},M)&= -\xi P_{i}.
\end{align*}
Now, let us use the $2$-cocycle condition for the elements $P_{i},P_{j},M_{1,2}\in V$ and substitute the values of $\varphi,$ then we have
\begin{align*}
&[P_{i},\varphi(P_{j},M_{1,2})]+[P_{j},\varphi(M_{1,2},P_{i})]+[M_{1,2},\varphi(P_{i},P_{j})]\\
&+\varphi(P_{i},[P_{j},M_{1,2}])+\varphi(P_{j},[M_{1,2},P_{i}])+\varphi(M_{1,2},[P_{i},P_{j}])\\
&=2\alpha[P_i,P_j]+\varphi(P_{i},Q_{j})-\varphi(P_{j},Q_{i})+b_i\varphi(M_{1,2},M)=0,
\end{align*}
which follows that $\rho=0$ and $2 b_i\alpha-2b_iy+b_i\varrho=0.$ Similarly, by $(d^2\varphi)(Q_{i},Q_{j},M_{1,2})=0,$ we get that $2b_i\mu+2b_iy+b_i\varrho=0.$ Thus, we obtain the restrictions $\varrho=2(\alpha+\mu), \ \rho=0, \ 2y=\alpha-\mu.$
Furthermore, the relation $(d^2\varphi)(P_{i},Q_{j},M_{1,2})=0$ yields to
\begin{align*}
\gamma[P_i,P_j]+\beta[Q_i,Q_j]-\varphi(P_{i},P_{j})+\varphi(Q_{j},Q_{i})=0.
\end{align*}
This follows that $\gamma+\beta+x+z=0.$ Thus, we get
\begin{align*}
\varphi(P_{i},P_{j})&= b_{i}\xi M_{1,2}+b_{j}xM, & \varphi(P_{i},Q_{j})&= \frac{1}{2}b_{j}(\alpha-\mu)M, & \varphi(Q_{i},Q_{j})&= b_i\xi M_{1,2}+b_jzM,\\
\varphi(P_{i},M_{1,2})&= \alpha P_{i}+\beta Q_{i} & \varphi(Q_{i},M_{1,2})&= \gamma P_{i}+\mu Q_{i}, & \varphi(M_{1,2},M)&= 2(\alpha+\mu) M,\\
\varphi(P_{i},M)&= \xi Q_{i} & \varphi(Q_{i},M)&= -\xi P_{i}.
\end{align*}
Hence, $\dim(Z^2(V,\widehat{\mathfrak{cga}}_\ell(2,\C))^{\ms})=6.$
If we apply the Hochschild-Serre factorization theorem, the assertion follows.
\end{proof}

\begin{rem}
Note that both conformal Galilei algebras $\mathfrak{cga}_\ell(2,\C)$ and its mass extension $\widehat{\mathfrak{cga}}_\ell(2,\C)$ are not perfect Lie algebras since their derived subalgebras do not contain the rotation generator $M_{1,2}.$
\end{rem}

\end{document}